\documentclass[11pt]{amsart}
\usepackage{geometry} % see geometry.pdf on how to lay out the page. There's lots.
\usepackage{amssymb}

\usepackage{pgfplots}
\usepackage{psfrag}
\usepackage{amsmath}
\usepackage{calc}
\usepackage{systeme}
\usepackage{amsfonts}
\usepackage{amsthm}
\usepackage{algorithm2e}
\usepackage[applemac]{inputenc}
\usepackage{bbold}
\usepackage[numbers]{natbib}

\title{Infinite Horizon Stochastic Optimal Control with sign-changing discount factor}
\author{Charles Bertucci \textsuperscript{a}, Jean-Michel Lasry\textsuperscript{a}, Pierre-Louis Lions\textsuperscript{a,b}
}
\thanks{\textsuperscript{A} Universit\'e Paris-Dauphine, PSL Research University,UMR 7534, CEREMADE, 75016 Paris, France, \\ \textsuperscript{B}Coll\`ege de France, 3 rue d'Ulm, 75005, Paris, France. \\
E-mail addresses: bertucci@ceremade.dauphine.fr , 2007lasry@gmail.com , pierre-louis.lions@college-de-france.fr}

\newtheorem{theorem}{Theorem}
\newtheorem{lemma}{Lemma}

\begin{document}

\maketitle
\begin{abstract}
We study an infinite horizon stochastic optimal control problem by means of the associated Hamilton-Jacobi-Bellman equation. The problem we are studying has the particularity of having a discount factor which can take both signs, depending on the value of the state.
\end{abstract}

\section{Introduction}
The present paper is concerned with the study of infinite horizon optimal control problems with a discount factor which is not of constant sign, as well as with the study of the associated Hamilton-Jacobi-Bellman (HJB) equation. In \cite{bll}, the authors studied a similar HJB equation arising from a model to characterize the equilibrium value of certain goods. A partial link with optimal control was made, and here, it is extended and be made more precise.

We consider the problem of controlling a process \((X_{t})_{t\geq 0}\), valued in a smooth connected domain \(\mathcal O\subset \mathbb{R}^{d}\), given as a solution of

\begin{equation}
\label{eq:process}
dX_t = \alpha_t dt + \sqrt{2} dB_t \quad t > 0,
\end{equation}
with reflection on $\partial \mathcal O$ and where \((\alpha_{t})_{t\geq 0}\) is the control and \((B_{t})_{t\geq 0}\) a standard d-dimensional Brownian motion on a probabilistic space \((\Omega, \mathcal{A}, \mathbb{P})\). The controller aims at minimizing the expectation of a discounted integral of the running cost \(\frac{1}{2} |\alpha_t|^2 + g(X_t)\), where \(g \in \mathcal{C}(\bar {\mathcal O})\), is a continuous function. The discount is made through a function \(r: \mathcal O\to \mathbb{R}\), which is not necessarily non-negative. Hence, the cost is given by

\begin{equation}
\label{eq:cost}
J(\alpha) := \mathbb{E} \left[ \int_0^{\infty} e^{-\int_0^t r(X_s) ds} \left[ \frac{1}{2} |\alpha_t|^2 + g(X_t) \right] dt \right].
\end{equation}

Clearly, without assumptions on \(r\), the previous expectation is not well defined for any control. We shall come back on this difficulty later on and, for the moment, we observe that, at least formally, the HJB equation associated to (\ref{eq:process})-(\ref{eq:cost}) is

\begin{equation}\label{eq:HJB}
\begin{cases}
r u - \Delta u + \frac{1}{2} |\nabla_x u|^2 = g & \text{in } \mathcal{O}, \\
\partial_\eta u = 0 & \text{on } \partial \mathcal{O},
\end{cases}
\end{equation}
where $\partial_\eta$ denotes derivative in the direction orthogonal to the boundary.

In \cite{bll}, the authors established that, when \(g \geq 0\), equation (\ref{eq:HJB}) admits a unique solution \(u\) which satisfies \(u(x) > 0\) for all \(x \in \mathcal{O}\) (Theorem 2), as well as natural stability properties, which are consequences of the comparison principle stated as follows. They also made a partial link with an optimal control problem.

\begin{lemma}
\label{lemma:comparison_positive}
Let \(u\) and \(v\) be two smooth functions, bounded from below by a positive constant such that

\[
\begin{cases}
r u - \Delta u + \frac{1}{2} |\nabla_x u|^2 \geq g & \text{in } \mathcal{O}, \\
r v - \Delta v + \frac{1}{2} |\nabla_x v|^2 \leq g & \text{in } \mathcal{O}, \\
\partial_\eta u = \partial_\eta v = 0 & \text{on } \partial \mathcal{O}.
\end{cases}
\]

If \(g \geq 0\) and \(\lambda_1(-\Delta + r Id) < 0\), then \(u \leq v\).
\end{lemma}

In the previous, \(\lambda_1(A)\) denotes the first eigenvalue of the operator \(A\) (when it is well defined).

Our main contributions, namely in respect to the results proven in \cite{bll} are: i) to extend Lemma \ref{lemma:comparison_positive} to cases in which we do not have \(g \geq 0\), thus proving an existence and uniqueness result to solutions to (\ref{eq:HJB}) with a certain bound from below (Theorem 3 below); ii) to make more complete the link with the optimal control problem, namely by identifying a set of controls \(\mathcal{A}\) such that \(\inf_{\alpha \in \mathcal A} J(\alpha)\) is the unique solution to (\ref{eq:HJB}).

We start first by working on the case \(g \geq 0\), before showing that we can in fact consider \(g > -c\) for a constant \(c > 0\) which depends on the spectral properties of \(-\Delta + r Id\).\\

Few works have dealt with infinite horizon stochastic optimal control with possible negative discount. We can mention \cite{sz} who dealt with discrete situations in which the discount will eventually remain positive after a certain random time, to \cite{goueletal} for pricing situations (i.e. no control) with negative interest rates and to \cite{toda} for another economic model with random interest rates.\\

\textbf{Modelling comment and consequences in economics:}\\
The function \(r\) is here naturally interpreted as a discount factor, but note that it can also be interpreted as the growth rate of a certain quantity (a mass for instance) which simply multiplies the running cost. Hence, the previous control problem can be seen as a control of a population with varying mass. Indeed, if we start from a mass $1$ for instance, that the population grows at rate $r(X_t)$, and that the instantaneous cost we are facing is proportional to the current mass, then we end up precisely with a problem of the form of \eqref{eq:cost}.

Not entering the economics modelling behind equation \eqref{eq:HJB}, we simply mention that Theorem \ref{theorem:uniqueness} has natural consequences in economics. In particular, it quantifies the size of the price (here $\max -g$) that the agents can rationally agree to pay to maintain the infrastructure of a reserve of value, like gold for instance.\\

In all what follows, we assume that \(\lambda_1(-\Delta + r \operatorname{Id}) < 0\) as well as the existence of a point $x_0 \in \mathcal O$ such that $r(x_0) > 0$. In particular, the sign of $r$ is not constant, because if $r$ is positive, then \(\lambda_1(-\Delta + r \operatorname{Id}) \geq 0\) from standard arguments.

\section{The Case of a Non-Negative Running Cost}

Let \(\mathcal{A} := \{a : \bar{\mathcal O} \to \mathbb{R}^d, \text{smooth}, \text{s.t. } \lambda_1(-\Delta - a \cdot \nabla_x + r \operatorname{Id}) > 0\}\).

The set \(\mathcal{A}\) is the natural candidate for the set of our controls. Indeed, remark that for a smooth \(a : \mathcal O \to \mathbb{R}^d\),

\begin{equation*}
\mathbb{E}\left[ e^{\int_0^Tr(X_t)dt}\right]  = O\left( e^{-\lambda_1(-\Delta - a \cdot \nabla_x + r \operatorname{Id}) T}\right),
\end{equation*}
where $(X_t)_{t \geq 0}$ is the solution of \eqref{eq:process} for $\alpha_t = a(X_t)$. See Lemma 4 in \cite{bll} for instance. Hence, \(\mathcal{A}\) is a set of feedback controls which guarantee that the cost is finite, given that $g \geq 0$. By assumptions  \( 0 \notin \mathcal{A}\), thus the controller indeed needs to do something to make its cost finite (if \(g > 0\)).

Furthermore, in terms of the associated partial differential equation (PDE), the set \(\mathcal{A}\) has nice properties which we recall in the following Lemma.

\begin{lemma}
\label{lemma:existence}
For all \(a \in \mathcal{A}, f \in \mathcal{C}(\bar{\mathcal O})\), the equation

\begin{equation}\label{eq:PDEa}
\begin{cases}
- \Delta u - a \cdot \nabla_xu + r u = f & \text{in } \mathcal O, \\
\partial_\eta u = 0 & \text{on } \partial \mathcal O,
\end{cases}
\end{equation}

has a unique solution. In addition, if \(f \geq 0\), \(f \neq 0\), then, for all \(x \in \mathcal O\), \(u(x) > 0\).
\end{lemma}

We do not recall the standard proof of such a result and invite the reader to have a look at \cite{bnv,pll} if she or he wants more details on the link between strong maximum principle, the existence of positive sub-solutions to the associated equation, and the first eigenvalue of the operator. Furthermore, recall that thanks to Feynman-Kac representation formulae, for $f \in \mathcal C(\bar {\mathcal O})$, the unique solution $u$ of \eqref{eq:PDEa} satisfies for any $x \in \mathcal O$
$$
u(x) = \mathbb E \left[\int_0^\infty e^{\int_0^t-r(X_s)ds} f(X_t)dt \right],
$$
where $(X_t)_{t \geq 0}$ is given as the sole strong solution of
$$
dX_t = a(X_t)dt + \sqrt{2}dW_t,
$$
with $X_0=x$ and reflection on the boundary of $\mathcal O$.

For any \(a \in \mathcal{A}\), we denote \(v_a\) the unique solution of \eqref{eq:PDEa} associated to \(f = \frac12 |a|^2\). The next result characterizes the unique positive solution \(u\) to (\ref{eq:HJB}) as the value of some optimal control problem.%, since, for any \(a \in \mathcal{A}\), \(v_a\) is the value associated to playing only \(\alpha_t = a(t, X_t)\) in (2), when \(g = 0\).
\begin{theorem}
Let $u$ be the unique positive solution to \eqref{eq:HJB} with $g = 0$ given by Theorem 1 in \cite{bll}. Then, $u = \inf_{a \in \mathcal A}v_a$.
\end{theorem}
\begin{proof}
Let \(a \in \mathcal{A}\). We have that for all \(x \in \mathcal O\), \(u(x) > 0\), \(\forall x \in \partial \mathcal O, \partial_\eta u(x) = 0\), as well as

\[
\begin{aligned}
r u(x) - \Delta u(x) - a \cdot \nabla_x u(x)& =  - \frac{1}{2} |\nabla_x u(x)|^2 - a \cdot \nabla_x u(x)\\
&\leq \frac12 |a|^2.
\end{aligned}
\]

Hence, using Lemma \ref{lemma:existence}, it follows that \(u \leq v_a\). Now observe that 
$$
r u(x) - \Delta u(x) - (-\nabla_x u) \cdot \nabla_x u(x)= \frac12 |\nabla_x u|^2 \geq 0.
$$
Recalling that $u$ is positive, we deduce from the previous that $\lambda_1(-\Delta - \nabla_x u \cdot \nabla_x + r \operatorname{Id}) > 0$. It follows that $-\nabla_x u \in \mathcal A$. Hence \(u = v_{-\nabla_x u}\) and the result follows.
\end{proof}

\section{More general running costs}

We now turn to the study of (\ref{eq:HJB}) and
$\inf_{\alpha \in \mathcal A} J(\alpha)$
when \(g\) can take negative values.

In \cite{bll}, the study of \eqref{eq:HJB} (for non-negative $g$) relied on two main arguments: the existence of a positive sub-solution to (\ref{eq:HJB}) and a comparison principle. We shall prove that there exists a constant \(c > 0\) depending on \(r\) such that, if \(g \in C(\bar{\mathcal O})\), \(g \geq -c\), then those two arguments can still be used.

We start with the subsolution property.

\begin{lemma}
\label{lemma:subsolution}
Let \(v_1\) be the first eigenfunction of \(-\Delta + r \mathrm{Id}\) normalized so that \(\int_{\mathcal O} \nu_1 = 1\), and \(\lambda_1 = \lambda_1(-\Delta + r \mathrm{Id})\). Then 
\begin{equation}\label{def:c}
-c := \inf_{\varepsilon > 0}\sup_{x \in \mathcal O} \left\{ \varepsilon \lambda_1 v_1(x) + \frac{1}{2} \varepsilon^2 |\nabla v_1(x)|^2 \right\} < 0,
\end{equation}
 and for \(g \geq -c\), there exists a subsolution \(v\) to \eqref{eq:HJB} such that \(\forall x \in \mathcal O, v(x) > 0\).
\end{lemma}

\begin{proof}
By construction \(-\Delta v_1 + r v_1 = \lambda_1 v_1\). Thus,

\begin{equation*}
  - \Delta(\varepsilon v_1) + r(\varepsilon v_1) + \frac{1}{2} |\nabla_x(\varepsilon v_1)|^2 = \varepsilon \lambda_1 v_1(x) + \frac{1}{2} \varepsilon^2 |\nabla_x v_1(x)|^2 + g(x) - g(x).
\end{equation*}

Now, since \(\forall x \in \bar{ \mathcal O}, v_1(x) > 0, c > 0\) holds. Furthermore, since \(v_1\) is not constant (because \(r\) is not either), \(c < +\infty\) and the infimum in \(\varepsilon\) is reached at some $\bar \varepsilon$. Hence, if \(g \geq -c\), then \(\bar \varepsilon v_1\) is a suitable subsolution.
\end{proof}

The previous Lemma is merely a computation, but it allows us to define precisely the constant \(c\) which is going to be of interest in the following. Furthermore, we denote by $\tilde w$ the subsolution $\bar\varepsilon v_1$. We now prove the following crucial comparison principle.

\begin{lemma}
\label{lemma:comparison}
Let \(c > 0\) be defined in (5) and consider \(g \geq -c\), \(g \in \mathcal{C}(\mathcal O)\), \(u, v\) smooth functions such that, $u \geq \tilde w$, for all $x \in \bar{\mathcal O}, v(x) > \tilde w(x)$, and 

\begin{equation*}
- \Delta u + \frac{1}{2} |\nabla u|^2 + r u \leq g \quad \text{in } \mathcal O,
\end{equation*}

\begin{equation*}
- \Delta v + \frac{1}{2} |\nabla v|^2 + r v \geq g \quad \text{in } \mathcal O,
\end{equation*}

\begin{equation*}
\partial_\eta u = \partial_\eta v = 0 \quad \text{on } \partial \mathcal O.
\end{equation*}

Then \(u \leq v\).
\end{lemma}

\begin{proof}
The proof is an adaptation of the proof of Lemma 1 in \cite{bll}, which is itself inspired from a technique of Laetsch \cite{laetsch}.
The objective is to compare \(\bar{u} = u - \tilde w\) and \(\bar{v} = v - \tilde w\).

Observe that

\begin{equation*}
- \Delta \bar{u} + r \bar{u} + \frac{1}{2} |\nabla \bar{u}|^2 \leq g + \Delta \tilde w - r\tilde  w - \frac{1}{2} |\nabla \tilde w|^2 - \nabla \tilde w \cdot \nabla \bar u,
\end{equation*}

\begin{equation*}
- \Delta \bar{v} + r \bar{v} + \frac{1}{2} |\nabla \bar{v}|^2 \geq g + \Delta \tilde w - r \tilde w - \frac{1}{2} |\nabla \tilde w|^2 - \nabla \tilde w \cdot \nabla \bar v.
\end{equation*}

Note, \(f := g + \Delta \tilde w - r\tilde w - \frac{1}{2} |\nabla \tilde w|^2 \geq 0\), and we rewrite the previous inequalities as

\begin{equation}\label{equbar}
- \Delta \bar{u} + r \bar{u} + \frac{1}{2} |\nabla \bar{u}|^2 + \nabla \tilde w \cdot \nabla \bar{u} \leq f,
\end{equation}

\begin{equation}\label{eqvbar}
- \Delta \bar{v} + r \bar{v} + \frac{1}{2} |\nabla \bar{v}|^2 + \nabla \tilde w \cdot \nabla \bar{v} \geq f.
\end{equation}

We are now in position to use Laetsch's technique. Define

\begin{equation*}
\theta := \max \left\{ \bar{\theta} \in (0, 1] \mid \bar{\theta} \bar{u} \leq \bar{v} \right\}.
\end{equation*}

Since \(\bar{u}\) and \(\bar{v}\) are respectively non-negative and positive, \(\theta > 0\). Furthermore, defining \(z = \bar{u}_{\theta} - \bar{v}\) where \(\bar{u}_\theta = \theta \bar{u}\), we compute $\theta\eqref{equbar} - \eqref{eqvbar}$ to obtain 

\begin{equation}\label{eq5}
- \Delta z + r z + \nabla \tilde w \cdot \nabla z + \frac{1}{2} (\nabla \bar{u}_\theta + \nabla \bar{v}) \cdot \nabla z \leq (\theta - 1) f + \frac12(\theta^2-\theta)|\nabla \bar u|^2\leq 0.
\end{equation}

But \(z \leq 0\) and \(\partial_{\eta} z = 0\) on \(\partial \mathcal O\) implies that either \(z\) is negative or \(z\) is constant equal to 0, as a consequence of the strong maximum principle.

If \(z\) is negative, then \(\theta = 1\) and the result is proved.

If \(z \equiv 0\), then \eqref{eq5} is in fact an equality. In particular either $\theta = 1$ and the result is proved, or $f = |\nabla \bar u|^2 \equiv 0$, hence \(\bar{u}_\theta\) is constant, but then so is \(\bar{v}\), which is not possible since \(r\) changes sign.
\end{proof}

The previous comparison principle is the main ingredient for the uniqueness of solutions to (\ref{eq:HJB}). It is also a fundamental tool to prove existence, which also requires an a priori estimate, which is provided in \cite{bll} (Lemma 2). We have the following.

\begin{theorem}
\label{theorem:uniqueness}
For \(c > 0\) defined in \eqref{def:c} and \(g \in \mathcal{C}(\bar{\mathcal O})\), \(g \geq -c\), there exists a unique solution \(u\) to (\ref{eq:HJB}) such that \(u(x) > \tilde w(x)\) for all \(x \in \bar {\mathcal O}\).
\end{theorem}

\begin{proof}
Uniqueness follows from Lemma \ref{lemma:comparison}. For the existence, we argue as in \cite{bll} and consider \(K > 0\) such that \(K + r\) is a positive function. Define then a sequence \((u_n)\) with \(u_0 = \tilde w\) and \(u_{n+1}\), the unique solution to

\[
\begin{cases}
- \Delta u_{n+1} + r u_{n+1} + K u_{n+1} + \frac{1}{2} |\nabla_x u_{n+1}|^2 = K u_n + g & \text{in } \mathcal O, \\
\partial_\eta u_{n+1} = 0 & \text{on } \partial \mathcal O.
\end{cases}
\]

By the maximum principle (which can be used since $K$ is large enough), \(\forall x, u_1(x) > u_0(x)\), therefore, the maximum principle also yields that $u_2 \geq u_1$ and thus, by induction that \((u_n)_{n \geq 0}\) is an increasing sequence. Using the bound of Lemma 2 in \cite{bll}, we also deduce that it is bounded. Its limit \(u\) is a required solution of the problem.
\end{proof}

To conclude the study of our stochastic optimal control problem, we show that the function \(u\) given by Theorem \ref{theorem:uniqueness} can be represented as it was the case in the previous section. For any \(a \in \mathcal{A}\), denote \(w_a\) the unique solution to \eqref{eq:PDEa} associated to \(f = \frac{1}{2} a^2 + g\).

\begin{theorem}
\label{theorem:representation}
For \(c > 0\) defined in (5), let \(g \in \mathcal{C}(\bar{\mathcal O})\), \(g \geq -c\). Let \(u\) be the function given by Theorem \ref{theorem:uniqueness}, then \(u = \inf_{a \in \mathcal{A}} w_a\). 
\end{theorem}
\begin{proof}
Let $a \in \mathcal A$ and consider $w_a$. We want to use Lemma \ref{lemma:comparison} to prove that $u \leq w_a$. In order to do so, we simply need to verify that for any $x \in \bar{\mathcal O}, w_a(x) > \tilde w(x)$. But by construction
$$
\begin{aligned}
-\Delta& (w_a -\tilde w) + r(w_a -\tilde w) - a \cdot \nabla_x(w_a- \tilde w)\\
&\geq \frac12|a|^2 + \frac 12 |\nabla_x \tilde w|^2 + a \cdot \nabla_x \tilde w = \frac12|a +\nabla_x \tilde w|^2 \geq 0.
\end{aligned}
$$
Since $a \in \mathcal A$, we deduce that either $w_a = \tilde w$ and $a = -\nabla_x \tilde w$, or $w_a -\tilde w$ is positive. In the first case, we found out that $u= \tilde w= w_a$. In the other one, using Lemma \ref{lemma:comparison}, we obtain $u \leq w_a$. Hence, we indeed have $u \leq \inf_{a \in \mathcal A}w_a$. To obtain the equality, it remains to verify that $-\nabla_x u \in \mathcal A$. But this follows from the inequality

$$
r (u-\tilde w) - \Delta (u - \tilde w) - (-\nabla_x u) \cdot \nabla_x (u-\tilde w)\geq \frac12 |\nabla_x u|^2  - \nabla_x u\cdot \nabla_x \tilde w + \frac12|\nabla_x \tilde w|^2\geq 0.
$$
Recalling that $u-\tilde w$ is non-negative, it follows that either $u - \tilde w$ is non-negative and then $-\nabla_x u \in \mathcal A$, or that $u = \tilde w$ and thus also, $-\nabla_x u = -\nabla_x w \in \mathcal A$. Hence \(u = w_{-\nabla_x u}\) and the result follows.

\end{proof}

\section*{Acknowledgements}
The authors acknowledge a partial support from the Chair FDD (Institut Louis Bachelier). The first author's work is supported by the ERC project PaDiESeM.

\end{document}